\documentclass[preprint,12pt]{elsarticle}

\usepackage[T1]{fontenc}
\usepackage{lmodern}
\usepackage{amsmath,amssymb,amsthm,mathtools}
\usepackage{microtype}
\usepackage{enumitem}
\usepackage{array}
\usepackage{booktabs}
\usepackage{float}
\usepackage{tikz}
\usetikzlibrary{arrows.meta,positioning,calc,fit,decorations.pathreplacing}
\usepackage{hyperref}
\hypersetup{hidelinks}
\journal{}

\newtheorem{theorem}{Theorem}
\newtheorem{lemma}[theorem]{Lemma}
\newtheorem{proposition}[theorem]{Proposition}

\theoremstyle{definition}

\newtheorem{problem}[theorem]{Problem}
\theoremstyle{remark}
\newtheorem{remark}[theorem]{Remark}

\newcommand{\defi}{\operatorname{def}}
\newcommand{\CONF}{\textnormal{\textsc{Conformability}}}
\newcommand{\PTPthree}{\textnormal{\textsc{Perfect Triangle Packing}}\ensuremath{_{\omega=3}}}
\newcommand{\ov}{\overline}
\newcommand{\dunion}{\mathbin{\dot\cup}}

\begin{document}

\begin{frontmatter}

\title{Conformability is NP-complete, even on connected regular graphs}

\author[inst1,inst2]{J\'ozsef Pint\'er\corref{cor1}}
\ead{pinterj@edu.bme.hu}
\cortext[cor1]{Corresponding author.}

\affiliation[inst1]{organization={Department of Stochastics, Institute of Mathematics,
  Budapest University of Technology and Economics},
  addressline={Egry J\'ozsef utca 1},
  postcode={1111},
  city={Budapest},
  country={Hungary}}

\affiliation[inst2]{%
  organization={HUN-REN--BME Stochastics Research Group},
  addressline={Egry J\'ozsef utca 1},
  postcode={1111},
  city={Budapest},
  country={Hungary}}

\begin{abstract}
A graph $G$ is conformable if it admits a proper $(\Delta(G)+1)$-coloring in which, among the $\Delta(G)+1$ color classes including the empty ones, at most
\(\sum_{v\in V(G)}(\Delta(G)-d_G(v))\)
have parity different from that of $|V(G)|$. The complexity of deciding conformability was left open in recent work, and positive results for several graph classes had suggested that the problem might be polynomial-time solvable. We settle the general problem by proving that \textsc{Conformability} is NP-complete. Hardness holds even for connected regular graphs $G$ of odd order with independence number $\alpha(G)=3$ and maximum degree $\Delta(G)\ge |V(G)|/2$. In particular, NP-completeness persists when every color class is forced to have the parity of the order. The reduction starts from perfect triangle packing in graphs of clique number three, regularizes the source graph while preserving the relevant triangle packings, and then takes the complement. In the complement, conformable color classes correspond to odd cliques of the regularized graph; $K_4$-freeness restricts these cliques to singletons or triangles, and the number of available colors forces exactly the required number of disjoint triangles.
\end{abstract}


\end{frontmatter}

\section{Introduction}

All graphs in this paper are finite, simple and undirected.  For a graph $G$, let $\Delta(G)$ be its maximum degree and define its \emph{deficiency} by
\[
\defi(G)=\sum_{v\in V(G)}(\Delta(G)-d_G(v)).
\]
A proper vertex coloring $\varphi:V(G)\to\{1,\ldots,\Delta(G)+1\}$ is \emph{conformable} if at most $\defi(G)$ color classes, including empty classes, have parity different from that of $|V(G)|$.  The corresponding decision problem is:

\begin{problem}[\CONF{}]\leavevmode\par\smallskip
\begin{description}[
  style=sameline,
  font=\normalfont\bfseries,
  widest=Question:,
  leftmargin=!,
  labelsep=.7em,
  topsep=0pt,
  itemsep=.25em,
  parsep=0pt
]
\item[Instance:] A connected graph \(G\).
\item[Question:] Is \(G\) conformable?
\end{description}
\end{problem}

The connected-input convention follows the recent formulation of the problem~\cite{FariaNigroPreissmannSasaki2023}.  This convention matters: conformability depends on the global parameters $\Delta(G)$ and $|V(G)|$, so the problem does not decompose over connected components.  Our reduction produces connected instances.

Conformable colorings were introduced by Chetwynd and Hilton~\cite{ChetwyndHilton1988} in the study of the Total Coloring Conjecture.  A total coloring of $G$ assigns colors to vertices and edges so that adjacent or incident elements receive different colors.  Graphs with total chromatic number $\chi''(G)=\Delta(G)+1$ are called Type~1.  Chetwynd and Hilton observed that every Type~1 graph is conformable.  Thus conformability is a structural necessary condition for Type~1 total colorability, but it is defined purely as a vertex-coloring property.  Although determining the total chromatic number is NP-hard in general~\cite{SanchezArroyo1989}, this does not determine the complexity of the vertex-coloring obstruction captured by conformability.

Conformability also has a broader role in total-coloring theory.  It appears in the Conformability Conjecture and in related work on totally critical graphs and the Overfull Conjecture~\cite{HamiltonHiltonHind1999,HiltonHolroydZhao2001}.  High-degree regular graphs are a classical setting in which conformability-type conditions interact with Type~1 total colorability: Chetwynd, Hilton and Zhao obtained high-minimum-degree total-coloring results and a necessary-and-sufficient condition for certain odd-order regular graphs to be Type~1, and Chew later improved the high-degree odd-order regular case~\cite{ChetwyndHiltonZhao1991,Chew1996}.  Thus the dense regular regime reached by our reduction is one of the regimes in which conformability has historically been closest to total coloring.

The known evidence before this paper was largely positive.  Chetwynd and Hilton classified cycles and complete graphs~\cite{ChetwyndHilton1988}; Hilton and Hind proved that every non-conformable graph $G$ with $\Delta(G)\ge2$ satisfies $\defi(G)\le\Delta(G)-2$, with the sharper bound $\defi(G)\le\Delta(G)-3$ when $|V(G)|$ is even and $\Delta(G)$ is odd~\cite[Lemma~9]{HiltonHind2002}; powers of cycles were classified with respect to conformability by Zorzi, de Figueiredo, Machado, Zatesko and Souza~\cite{ZorziFigueiredoMachadoZateskoSouza2022}; and regular line graphs have also been studied from the conformability viewpoint~\cite{FariaNigroSasakiLineGraphs2023}.

Algorithmic evidence also pointed toward tractability.  Equitable coloring, guaranteed by the theorem of Hajnal and Szemer\'edi~\cite{HajnalSzemeredi1970} and constructible in polynomial time~\cite{KiersteadKostochkaMydlarzSzemeredi2010}, has been used to construct large conformable families~\cite{FariaNigroPreissmannSasaki2023}.  Recent work on circulant graphs formulated \CONF{} explicitly as a decision problem, stated that its complexity was unknown, developed clues toward tractability through equitable-coloring constructions and infinite positive classes, and concluded by conjecturing that conformability is polynomial-time solvable~\cite{FariaNigroPreissmannSasaki2023}.  A subsequent paper proved polynomial-time algorithms for regular bipartite graphs and for subcubic graphs, while again leaving the general problem open~\cite{FariaNigroSasaki2025}.  This separation from total coloring is meaningful: total coloring is NP-hard even on $k$-regular bipartite graphs for each fixed $k\ge3$~\cite{McDiarmidSanchezArroyo1994}, whereas ordinary conformability is polynomial-time solvable on regular bipartite graphs~\cite{FariaNigroSasaki2025}.

We distinguish this problem from a recent strengthening.  Faria, Nigro and Sasaki introduced strong conformable colorings, which are conformable colorings satisfying additional conditions designed to characterize extendability to a $(\Delta+1)$-total coloring~\cite{FariaNigroSasakiStrong2025}.  Hardness for that strengthened problem does not settle \CONF{} itself: ordinary conformability does not ask whether the vertex coloring extends to a total coloring.

Our result settles the original decision problem.

\begin{theorem}\label{thm:main}
The decision problem \CONF{} is NP-complete.  Moreover, it remains NP-complete when restricted to connected regular graphs $G$ of odd order satisfying $\alpha(G)=3$ and $\Delta(G)\ge |V(G)|/2$.
\end{theorem}

Our contribution is fourfold.
\begin{itemize}[leftmargin=2em]
\item We prove \CONF{} NP-complete, settling the decision problem.
\item We show that hardness already holds in the zero-deficiency regime of odd-order regular graphs.
\item We obtain dense connected regular hard instances with independence number exactly three.
\item We introduce a triangle-preserving regularization lemma for $K_4$-free graphs.
\end{itemize}

The restriction in Theorem~\ref{thm:main} is strong.  If $G$ is regular, then $\defi(G)=0$.  Hence, when $|V(G)|$ is odd, a conformable coloring is exactly a proper $(\Delta(G)+1)$-coloring in which every one of the $\Delta(G)+1$ color classes is odd.  Empty color classes are excluded automatically.  Hardness thus appears already at zero deficiency.

This also sharpens the contrast with the known positive results. The known polynomial-time cases include sparse graphs, such as subcubic graphs, as well as bipartite regular graphs. Our hard instances lie on the opposite side in two respects. First, they are dense. In the reduction we construct a regular auxiliary graph $H$ on $N$ vertices, say of degree $r$, and arrange that $N-1-r\ge N/2$. The output graph is $G=\overline H$, so $G$ is regular of degree $N-1-r$, and hence has degree at least half its order. Second, the hard instances have bounded independence number, namely $\alpha(G)=3$; in contrast, every bipartite graph of order $N$ has an independent set of size at least $N/2$. Thus regularity alone does not determine which cases are polynomial-time and which are NP-complete; density and small independence number are also relevant.

The core of the reduction is simple. Starting from a perfect triangle-packing instance in a $K_4$-free graph, we construct a regular $K_4$-free graph $H$ without creating any new triangle, and then set $G=\ov H$.  Color classes of $G$ are independent sets of $G$, equivalently cliques of $H$.  Since $H$ is $K_4$-free, every odd nonempty color class in a conformable coloring of $G$ has size either $1$ or $3$.  A counting identity then forces exactly the desired number of size-three classes, and these classes are precisely disjoint triangles of $H$.

The paper is organized as follows.  Section~\ref{sec:background} fixes notation and states the perfect triangle-packing problem.  Section~\ref{sec:padding} performs a small padding step that aligns the number of required triangles with the later degree parameter.  Section~\ref{sec:regularization} proves the triangle-preserving regularization lemma.  Section~\ref{sec:np-proof} completes the complement reduction and proves Theorem~\ref{thm:main}.  Section~\ref{sec:conclusion} discusses consequences and open directions.

\section{Background and notation}\label{sec:background}

We write $\omega(G)$ for the clique number and $\alpha(G)$ for the independence number of $G$.  The complement of $G$ is denoted by $\ov G$.  Unless otherwise stated, a color class may be empty; empty color classes are included when testing conformability.

The following elementary observation is the bridge between conformability and triangle packing.

\begin{proposition}\label{prop:regular-odd}
Let $G$ be a regular graph of odd order $N$.  Then $G$ is conformable if and only if $G$ has a proper $(\Delta(G)+1)$-coloring in which all $\Delta(G)+1$ color classes are nonempty and odd.
\end{proposition}

\begin{proof}
Since $G$ is regular, $\defi(G)=0$.  Therefore no color class may have parity different from $N$.  Since $N$ is odd, every color class must be odd, and in particular no color class is empty.  Conversely, any proper $(\Delta(G)+1)$-coloring whose color classes are all odd has no parity-defective class and is therefore conformable.
\end{proof}

We reduce from the following perfect-packing problem, also called \emph{Triangle Cover} or \emph{Partition Into Triangles} in the literature.

\begin{problem}[\PTPthree{}]\label{prob:ptp}
\leavevmode\par\smallskip
\begin{description}[
  style=sameline,
  font=\normalfont\bfseries,
  widest=Question:,
  leftmargin=!,
  labelsep=.7em,
  topsep=0pt,
  itemsep=.25em,
  parsep=0pt
]
\item[Instance:] A graph \(X_0\) with \(\omega(X_0)=3\) and \(|V(X_0)|=3p_0\).
\item[Question:] Can \(V(X_0)\) be partitioned into \(p_0\) triangles?
\end{description}
\end{problem}

\begin{theorem}[Perfect triangle packing at clique number three~\cite{GuruswamiRanganChangChangWong1998}]\label{thm:ptp-source}
The problem \PTPthree{} is NP-complete.
\end{theorem}

Guruswami et al.~\cite{GuruswamiRanganChangChangWong1998} define triangle cover as the existence of a perfect triangle packing and explicitly record NP-completeness on graphs of clique number $3$.  Garey and Johnson list the classical source problem as GT11; their Chapter~3 reduction is Theorem~3.7, and the comment under GT15 records that this construction contains no $K_4$~\cite[Theorem~3.7 and pp.~192--193]{GareyJohnson1979}.  We use this exact restricted formulation.  The condition $|V(X_0)|=3p_0$ is part of the perfect-packing problem itself, so no divisibility padding is needed.  Moreover, $\omega(X_0)=3$ is equivalent to $K_4$-freeness together with the existence of a triangle; in particular $p_0\ge1$.

\section{Padding the source instance}\label{sec:padding}

Let $X_0$ be an instance of Problem~\ref{prob:ptp}, so $|V(X_0)|=3p_0$.  The padding step increases the target number $p$ until the even degree parameter $d:=2p$ dominates the maximum degree of the padded graph, while preserving the yes/no answer.  Later, the identity $p=d/2$ is exactly what the complement counting argument forces.

Choose an integer $q\ge 1$ such that
\[
2(p_0+q)\ge \Delta(X_0).
\]
For instance, $q=\max\{1,\lceil \Delta(X_0)/2\rceil-p_0\}$ suffices, so this is a polynomial-time choice.  Define
\[
X:=X_0\dunion qK_3,\qquad p:=p_0+q,\qquad d:=2p.
\]
Then $|V(X)|=3p$, $d$ is even, and $X$ is still $K_4$-free.  Furthermore,
\[
\Delta(X)=\max\{\Delta(X_0),2\}.
\]
The choice of $q$ gives $d\ge\Delta(X_0)$, while $q\ge1$ gives $d\ge2$; hence $d\ge\Delta(X)$.

\begin{lemma}\label{lem:padding}
The graph $X_0$ has a perfect triangle packing if and only if $X$ contains $p$ pairwise vertex-disjoint triangles.
\end{lemma}

\begin{proof}
If $X_0$ has $p_0$ pairwise vertex-disjoint triangles, then together with the $q$ added $K_3$ components we obtain $p_0+q=p$ pairwise vertex-disjoint triangles in $X$.

Conversely, suppose $X$ has $p=p_0+q$ pairwise vertex-disjoint triangles.  The $q$ added components together contain at most $q$ vertex-disjoint triangles, and $X_0$ has only $3p_0$ vertices, so it contributes at most $p_0$ vertex-disjoint triangles.  Both upper bounds must therefore be tight.  Hence $X_0$ contributes exactly $p_0$ vertex-disjoint triangles and has a perfect triangle packing.
\end{proof}

\section{Regularizing without creating triangles}\label{sec:regularization}

The technical goal is to replace the padded $K_4$-free graph $X$ by a regular $K_4$-free graph $H$ while preserving its triangles exactly.  The three requirements are all necessary for the complement step: regularity makes the deficiency zero, odd order makes conformable color classes odd, and $K_4$-freeness ensures that odd cliques in $H$ have size only $1$ or $3$.  The intermediate graph $H$ need not be connected; only the final instance $G=\overline H$ must be connected, and this will follow from its minimum degree.

We first isolate the matching fact used in the construction.

\begin{lemma}\label{lem:half-dense-bipartite}
Let $Q$ be a balanced bipartite graph with bipartition $(A,B)$, where $|A|=|B|=L$.  If $\delta(Q)\ge L/2$, then $Q$ has a perfect matching.
\end{lemma}

\begin{proof}
We verify Hall's condition.  Let $S\subseteq A$.  If $S=\emptyset$, the condition is trivial.  If $0<|S|\le L/2$, then any vertex of $S$ has at least $L/2$ neighbors, so $|N(S)|\ge L/2\ge |S|$.  If $|S|>L/2$ and $N(S)\ne B$, choose $b\in B\setminus N(S)$.  Then all neighbors of $b$ lie in $A\setminus S$, so
\[
 d_Q(b)\le |A\setminus S|<L/2,
\]
contradicting $\delta(Q)\ge L/2$.  Hence $N(S)=B$, and again $|N(S)|=L\ge |S|$.  Hall's theorem gives a perfect matching.
\end{proof}

We now carry out the regularization step.  The point is to make all degrees equal to $d$ without introducing any new triangles. The construction keeps the original vertices of $X$ as \emph{centers}.  Whenever a center has degree below $d$, we attach new degree-one vertices, called \emph{ports}, to supply its missing degree.  We then add a small number of dummy centers with their own ports, both to provide enough flexibility for the matching construction and to fix the parity of the final order.

The remaining task is to raise every port from degree $1$ to degree $d$.  This is done by splitting the ports into two equal parts and adding $d-1$ edge-disjoint perfect matchings between them, while forbidding edges between ports belonging to the same center.  These forbidden pairs are precisely what prevent the creation of a triangle using one center and two of its ports; the bipartite nature of the added port graph prevents triangles using only ports.  Thus the construction regularizes $X$ while preserving its triangle set exactly.

\begin{lemma}[Triangle-preserving regularization]\label{lem:regularization}
Let $X$ be a $K_4$-free graph with $|V(X)|=3p$, and let $d=2p\ge \Delta(X)$.  One can construct, in polynomial time, a simple $d$-regular $K_4$-free graph $H$ of odd order such that, under the natural embedding $V(X)\subseteq V(H)$, the triangle vertex-sets of $H$ are exactly those of $X$.  Moreover, $H$ satisfies $|V(H)|\ge 2d+2$.  Consequently,
\begin{align*}
&X\text{ has }p\text{ pairwise vertex-disjoint triangles}\\
&\hspace{2.5cm}\Longleftrightarrow
H\text{ has }d/2\text{ pairwise vertex-disjoint triangles}.
\end{align*}
\end{lemma}

\begin{proof}
Since $d=2p$, the degree parameter $d$ is even.  We keep all vertices and edges of $X$.  The vertices inherited from $X$ are called \emph{centers}.  For each center $v\in V(X)$, add $d-d_X(v)$ new vertices, called \emph{ports of $v$}, and join each of them only to $v$ for the moment.  Since $d\ge\Delta(X)$, this number is nonnegative.  At this stage every original center has degree $d$, while each port has degree $1$.

We next add dummy centers.  A dummy center has no edges to other centers and receives exactly $d$ ports adjacent to it.  If $h$ dummy centers are added, let
\[
R(h)=\sum_{v\in V(X)}(d-d_X(v))+hd
\]
be the total number of ports.

\emph{Claim 1: $R(h)$ is even for every integer $h$.}  Since $d$ is even, both $|V(X)|d$ and $hd$ are even.  Also $\sum_{v\in V(X)}d_X(v)=2|E(X)|$ is even.  Therefore
\[
R(h)=|V(X)|d-2|E(X)|+hd
\]
is even.  Write $R(h)=2L$.

\emph{Claim 2: we may choose $h\in\{8,9\}$ so that the final order is odd and $L\ge 4d$.}  Let
\[
\nu(h):=|V(X)|+h+R(h).
\]
Because
\[
\nu(9)-\nu(8)=d+1
\]
and $d+1$ is odd, exactly one of $\nu(8)$ and $\nu(9)$ is odd.  Choose $h\in\{8,9\}$ giving odd $\nu(h)$, and set $N=\nu(h)$.  Since $h\ge 8$,
\[
R(h)\ge 8d,
\qquad\text{hence}\qquad
L=R(h)/2\ge 4d.
\]
Also
\[
N=|V(X)|+h+R(h)\ge |V(X)|+8+8d\ge 2d+2.
\]
The consecutive choices 8 and 9 are convenient; any sufficiently large consecutive pair would work. They provide enough slack for $L\ge4d$ while letting us fix the parity of the order by changing $h$ by one.

Split the $2L$ ports arbitrarily into two sets $A$ and $B$, each of size $L$.  This split need not respect centers; the auxiliary graph will simply forbid cross-pairs of ports belonging to the same center.  Define an auxiliary bipartite graph $Q$ with bipartition $(A,B)$ by joining $a\in A$ to $b\in B$ unless $a$ and $b$ are ports of the same center.  Each center has at most $d$ ports, so each port has at most $d$ forbidden vertices on the opposite side.  Hence
\[
\delta(Q)\ge L-d.
\]

\emph{Claim 3: $Q$ contains $d-1$ pairwise edge-disjoint perfect matchings.}  Suppose $s<d-1$ perfect matchings have already been removed.  Since each removed perfect matching uses exactly one edge incident with each vertex, the remaining bipartite graph has minimum degree at least
\[
L-d-s\ge L-d-(d-2)=L-2d+2.
\]
As $s\le d-2$ and $L\ge4d$, we have $L-2d+2\ge L/2$.  Lemma~\ref{lem:half-dense-bipartite} therefore gives a perfect matching in the remaining graph.  At each iteration we choose such a matching in the current remaining graph and delete its edges before continuing, so no edge is selected twice.  After $d-1$ iterations the resulting perfect matchings are pairwise edge-disjoint.  Each iteration is polynomial-time; for example, the Hopcroft--Karp algorithm is sufficient~\cite{HopcroftKarp1973}.

Let $F$ be the union of these $d-1$ matchings, and add the edges of $F$ among the ports.

\emph{Claim 4: the resulting graph $H$ is simple and $d$-regular.}  The graph is simple because the center-center edges are exactly the simple edges of $X$, every center-port edge joins a port to its unique center, and the port-port edges are the simple edges of the union of edge-disjoint matchings.  Every center has degree $d$ by construction.  Every port has one incident center-port edge and one edge in each of the $d-1$ perfect matchings, so every port has degree $1+(d-1)=d$.

\emph{Claim 5: no new triangle is created.}  A triangle on three centers is exactly a triangle of $X$.  A triangle with two centers and one port is impossible because each port is adjacent to only one center.  A triangle with one center $c$ and two ports would require both ports to be ports of $c$; such ports are not adjacent, because if they lie on the same side of $(A,B)$ then $F$ has no edge between them, and if they lie on opposite sides then their edge was excluded from $Q$.  Finally, three ports cannot form a triangle because $F$ is bipartite.  Consequently, every triangle of $H$ lies entirely in the original center set and is exactly a triangle of $X$.

\emph{Claim 6: $H$ is $K_4$-free.}  A $K_4$ using only centers would already be a $K_4$ in $X$.  A $K_4$ containing a new vertex would contain a triangle containing a new vertex, contradicting Claim~5.  Hence $H$ is $K_4$-free.

It remains only to record the size of the construction.  If $n=|V(X)|$, then $d=2p=2n/3=O(n)$.  The number of original ports is at most $nd$, and the number of dummy centers and their ports is at most $9+9d$.  Therefore
\[
|V(H)|\le n+nd+9+9d=O(n^2).
\]
The auxiliary graph $Q$ has polynomial size and the construction of the $d-1=O(n)$ matchings is polynomial.  Finally, the equivalence follows from $d/2=p$ and from Claim~5, which says that the triangle sets of $X$ and $H$ are identical.
\end{proof}

\begin{remark}
The regularization lemma may be useful beyond the present reduction.  It turns an arbitrary $K_4$-free perfect triangle-packing instance into a regular $K_4$-free graph while preserving the triangle set exactly, and it does so using only bipartite matching among newly added ports.
\end{remark}

\section{The NP-completeness proof}\label{sec:np-proof}

We now complete the reduction.  Starting from $X_0$, form $X$ as in Lemma~\ref{lem:padding}.  Apply Lemma~\ref{lem:regularization} to obtain a $d$-regular $K_4$-free graph $H$ of odd order $N\ge 2d+2$.  Finally, define
\[
G:=\ov H.
\]
Because $H$ is $d$-regular on $N$ vertices, $G$ is regular of degree
\[
\Delta(G)=N-1-d.
\]
Thus
\[
\Delta(G)+1=N-d.
\]
Since $G$ is regular, $\defi(G)=0$.  Since $N$ is odd, Proposition~\ref{prop:regular-odd} says that conformability of $G$ is equivalent to the existence of a proper coloring with exactly $N-d$ nonempty odd color classes.

Moreover, $q\ge1$ ensures that $X$ contains a triangle, and the regularization preserves the triangle set.  Since $H$ is also $K_4$-free,
\[
\omega(H)=3
\qquad\text{and hence}\qquad
\alpha(G)=3.
\]
The graph $G$ is connected.  Indeed,
\[
\delta(G)=N-1-d\ge N/2,
\]
where the inequality follows from $N\ge 2d+2$.  If $G$ were disconnected, some component would have size at most $N/2$, and a vertex in that component would have degree at most $N/2-1$, contradicting $\delta(G)\ge N/2$.

\begin{lemma}\label{lem:equivalence}
The graph $G$ is conformable if and only if $H$ contains $d/2$ pairwise vertex-disjoint triangles.
\end{lemma}

\begin{proof}
Suppose first that $G$ is conformable.  Since $G$ is regular and $N$ is odd, Proposition~\ref{prop:regular-odd} implies that all $N-d=\Delta(G)+1$ color classes are nonempty and odd.  Each color class is an independent set in $G$, hence a clique in $H$.  Since $H$ is $K_4$-free, every such clique has size at most $3$.  Therefore every color class has size either $1$ or $3$.

Let $t$ be the number of color classes of size $3$.  The remaining $N-d-t$ color classes have size $1$.  Counting vertices gives
\[
N=3t+(N-d-t)=N-d+2t.
\]
Hence $2t=d$, so $t=d/2$.  Each size-three color class is a triangle of $H$, and the color classes are pairwise disjoint.  Thus $H$ has $d/2$ pairwise vertex-disjoint triangles.

Conversely, suppose $H$ has $d/2$ pairwise vertex-disjoint triangles.  In $G=\ov H$, each such triangle is an independent set of size $3$.  Assign one color to each of these triangles, and assign a fresh singleton color to every remaining vertex.  The number of colors used is
\[
\frac d2+\left(N-3\frac d2\right)=N-d=\Delta(G)+1.
\]
All color classes have odd size, either $3$ or $1$.  Since $G$ is regular and $N$ is odd, Proposition~\ref{prop:regular-odd} gives a conformable coloring of $G$.
\end{proof}

\begin{proof}[Proof of Theorem~\ref{thm:main}]
Membership in NP is immediate.  A certificate is a map $\varphi:V(G)\to\{1,\ldots,\Delta(G)+1\}$; in polynomial time one checks properness, computes the color-class sizes, computes $\defi(G)$, and counts the classes whose parity differs from that of $|V(G)|$.

For NP-hardness, reduce from \PTPthree{}.  Given $X_0$, construct $X$, $H$ and $G=\ov H$ as above.  The construction is polynomial by Lemma~\ref{lem:regularization}.  By Lemma~\ref{lem:padding}, $X_0$ has a perfect triangle packing if and only if $X$ has $p$ pairwise vertex-disjoint triangles.  By Lemma~\ref{lem:regularization}, this is equivalent to $H$ having $d/2=p$ pairwise vertex-disjoint triangles.  By Lemma~\ref{lem:equivalence}, this is equivalent to $G$ being conformable.

The output graph $G$ is connected, regular, of odd order, satisfies $\alpha(G)=3$, and has $\Delta(G)=N-1-d\ge |V(G)|/2$.  Therefore \CONF{} is NP-hard even on the restricted class stated in the theorem.  Since \CONF{} belongs to NP, it is NP-complete.
\end{proof}

\section{Conclusion}\label{sec:conclusion}

We proved that \CONF{} is NP-complete, resolving the general complexity question for conformable colorings.  The hardness holds even for connected regular graphs of odd order with independence number exactly three and degree at least half the order.  In this regime the deficiency is zero, so every color class must have exactly the parity of the order; the hardness therefore comes from the parity structure of $(\Delta+1)$-colorings, not from special classes permitted by positive deficiency.

The result complements the known polynomial-time classifications.  Regular bipartite graphs and subcubic graphs are tractable, whereas dense connected regular graphs with $\alpha=3$ are NP-complete. Conformability thus has both polynomial-time and NP-complete cases, with the boundary between them still open for several natural graph classes.

Several questions remain open.
\begin{itemize}[leftmargin=2em]
\item Does NP-hardness persist under additional sparsity assumptions, or for bounded maximum degree $\Delta\ge4$?
\item What is the complexity on structured graph classes such as cographs, claw-free graphs, chordal graphs, planar graphs, or line graphs beyond the known regular-line-graph results?
\item What is the correct complexity for graphs with $\alpha(G)\le2$?  In the regular odd zero-deficiency regime our mechanism cannot apply, since every odd color class would be a singleton; positive-deficiency or nonregular inputs may behave differently.
\end{itemize}

These questions would refine the boundary between the positive algorithmic results and the hardness proved here.

\section*{Funding}
J\'ozsef Pint\'er is funded by Project No.~KDP-\allowbreak IKT-\allowbreak 2023-\allowbreak 900-\allowbreak I1-\allowbreak 00000957/\allowbreak 0000003 with support provided by the Ministry of Culture and Innovation of Hungary from the National Research, Development and Innovation Fund, financed under the KDP-2023 funding scheme. The funder had no role in the design of the study, the preparation of the manuscript, or the decision to submit the article for publication.

\section*{Declaration of competing interest}
The author declares that he has no competing interests.

\section*{Data availability}
No datasets were generated or analyzed during the current study.

\section*{Declaration of generative AI and AI-assisted technologies in the manuscript preparation process}
During the preparation of this work, the author used OpenAI's ChatGPT and Anthropic's Claude in order to support brainstorming, organization, \LaTeX{} drafting, and editorial revision. After using these tools, the author reviewed and edited the content as needed and takes full responsibility for the content of the published article.

\section*{Acknowledgments}
The author thanks Lu\'erbio Faria, Mauro Nigro, Myriam Preissmann, and Diana Sasaki, whose work on conformability brought the complexity question to the foreground and provided the starting point for this paper.

\bibliographystyle{abbrvurl}
\bibliography{references}

\end{document}